\documentclass{monsky2009} 

\usepackage{bm}

\usepackage{amssymb}
\usepackage{hyperref}

\newenvironment{conjecture*}[1][]{\textbf{Conjecture #1\hspace{.3em}}}{}
\newenvironment{theorem*}[1]{\textbf{#1}\itshape \hspace{.3em}}{\upshape}
\newenvironment{example*}[1]{\textbf{#1}\itshape \hspace{.3em}}{\upshape}
\newenvironment{observation*}[1][]{\textbf{Observation #1}\itshape \hspace{.3em}}{\upshape} 
\newenvironment{definition*}[1][]{\textbf{Definition #1}\itshape \hspace{.3em}}{\upshape} 
\newenvironment{corollary*}[1][]{\textbf{Corollary #1\hspace{.3em}}}{} 
\newenvironment{lemma*}[1][]{\textbf{Lemma #1\hspace{.3em}}}{} 
\newenvironment{remark*}[1][]{\textbf{Remark #1\hspace{.3em}}}{} 
\newenvironment{proof}[1][]{\textbf{Proof #1\hspace{.3em}}}{}
\newenvironment{proofs}{\textbf{Proofs\hspace{.3em}}}{}

\newtheorem{definition}{Definition}[section]
\newtheorem{theorem}{Theorem}[section] 

\newcounter{kpremark}
\newcommand{\ord}{\ensuremath{\mathrm{ord}}}

\newcommand{\mod}[1]{\ensuremath{\hspace{.5em}(#1)}}

\newcommand{\newo}{\ensuremath{\mathcal{O}}}
\newcommand{\newk}{\ensuremath{\mathcal{K}}}
\newcommand{\newc}{\ensuremath{\mathbb{C}}}

\newcommand{\newq}{\ensuremath{\mathbb{Q}}}
\newcommand{\newz}{\ensuremath{\mathbb{Z}}}

\newcommand{\nequiv}{\ensuremath{\equiv\hspace{-1em}/\ \,}}

  \renewcommand{\theenumi}{\alph{enumi}}

\begin{document}

\addtolength{\parskip}{-2pt}
\begin{frontmatter}






\title{The sum of two cubes problem---an approach that's classroom friendly \\[2ex] \small (For Al and Micky Cuoco)}
\author{Paul Monsky}

\address{Brandeis University, Waltham MA  02454-9110, USA. monsky@brandeis.edu}

\begin{abstract}
In this note I give simple proofs of classical results of Euler, Legendre and Sylvester showing that for certain integers $M$ there are no (or only a few) solutions of $x^{3} + y^{3} = M$, with $x$ and $y$ in \newq. The proofs all use a single argument—infinite 3-descent in the ring $\newo = \newz[\omega]$ of Eisenstein integers. (Everything needed about \newo\ is developed from scratch.) The reader only needs the briefest acquaintance with complex numbers, fields and congruence modulo an element of a commutative ring. In particular I never say anything about ideals or elliptic curves (though I do mention cubic reciprocity in passing), and a clever high-school student might well enjoy the note. A few new results with $M$ in \newo\ and $x$ and $y$ in $\newq[\omega]$ are also derived.
\end{abstract}


\end{frontmatter}

\renewcommand{\thesection}{\arabic{section}}
\setcounter{section}{-1} 
\section{Introduction}
\label{section0}

The problem of whether it's possible to write a particular integer $M$ as a sum of two rational cubes has a long history. For an excellent account of it, covering events from the time of Euclid and Diophantus through 2019, see \cite{1}.  (There have been developments since that I'll mention.) 

Even when $M$ is prime, questions remain, though it now appears that they've been settled when $M\nequiv 1 \bmod 9$. (When $M\equiv 1 \mod{9}$ the results quoted in \cite{1} depend on a famous unproved conjecture.)

Diophantus knew that the equation $x^{3} + y^{3} = 7$ had a solution with $x$ and $y$ rational. For us there's the evident solution (2,-1), but neither Euclid nor Diophantus imagined such monstrosities as negative numbers or zero. Diophantus, perhaps influenced by Euclid, claimed that from two cubes whose difference is a given $M$, two cubes whose sum is $M$ can be produced, and in particular one can pass from $(2)^{3}- (1)^{3} = 7$  to $(\frac{4}{3})^{3}+(\frac{5}{3})^{3}=7$. For an imagined reconstruction of the argument in the language of the day see \cite{1}.

Whether Diophantus had any interest in producing more solutions is problematic, but the ``tangent line'' and ``secant line'' methods of Fermat had their roots in these earlier constructions, and can be used to show that unless $M$ is a cube or twice a cube, $x^{3} + y^{3} = M$ has no rational solutions or an infinitude of such solutions. (We'll see later what happens in the exceptional cases, i.e when $M$ is 1 or 2.)


In a famous annotation in his copy of Diophantus' ``Arithmetica'', Fermat claimed to have found a truly marvelous proof that for $n>2$ the equation $x^{n}+y^{n}=1$ has no solutions in \newq\ except for the evident ones with $xy=0$. This ``FLT($n$)'' was finally demonstrated in recent years but (using a technique of ``infinite descent'' developed by Fermat), Euler and Kummer long ago found proofs of FLT(3). Kummer's proof (and Euler's implicitly) involved working in the ring, \newo,  of Eisenstein integers, which contains all three of the complex cube roots of 1. Kummer in fact showed that the only solutions in the field of fractions of the ring had $x=0$, (so $y^{3}=1$) or $y=0$, (so $x^{3}=1$).

When $M=2$, Euler proved that the only solution of $x^{3}+y^{3}=M$ with $x$ and $y$ in \newq\ is $(1,1)$. Legendre showed there were no solutions when $M$ is 3, 4 or 5 and said this also held when $M$ is 6. (But here he was mistaken. $(\frac{37}{21},\frac{17}{21})$ is a solution, and as we've indicated, we can conclude that there are infinitely many.) The ideas of Euler and Legendre were taken up by Pépin, Sylvester and Lucas in the late 1800's and early 1900's. They explored (among many others) the cases where $M$ is $p$ or $p^{2}$, with $p$ prime, and showed there were no solutions when $p\equiv 2\ \mbox{or } 5 \bmod{9}$, except for $(1,1)$ when $M=2$.

In what follows I'm concerned with establishing all these non-existence results, and some generalizations with \newz\ replaced by \newo\ , and ``prime'' by ``irreducible in \newo\,''. All this originated when I found a nice proof by 3-descent of Euler's result on $x^{3}+y^{3}=2$, and used it on MathOverflow to find the solutions of $y^{2}=x^{3}+1$ in \newq.

At the same time I found a closely related proof of Euler's $x^{3}+y^{3}=1$ theorem. Earlier this year I realized that these very simple proofs generalized to give many of Sylvester's theorems, and that further attractive results emerged when one passed from \newz\ to \newo. All that's really needed is standard facts about \newo, which I'll develop from scratch in the following section. I've presented the arguments when $M$ is 1 or 2 in undergraduate number theory courses---they're accessible to students with an interest in number theory and algebra. I've avoided the language of ideals and elliptic curves, but assume the reader knows a little about rings, fields, and congruence mod an element.

\renewcommand{\thesection}{\Alph{section}}
\section{Simple properties of \newo\ and \newk  }
\label{sectionA}

\begin{definition}
\label{defA.1}
$u$ and $v$ in \newc\ are the roots of $z^{2}+z+1=0$. \newo\ and \newk\ are the \newq\ and \newz-linear combinations respectively of $1$, $u$ and $v$. There is a multiplicative function from \newc\ to the non-negative reals taking $w$ to $w\bar{w}=|w|^{2}$. The restriction of this function to \newk\ is the ``norm map'', $w\rightarrow N(w)$.
\end{definition}

Note that $uv=1$. Also $u+v=-1$, so any two of $1$, $u$ and $v$ furnish a \newz-basis of \newo\ and a \newq-basis of \newk.  Since $\frac{z^{3}-1}{z-1}=z^{2}+z+1$, $u^{3}$ and $v^{3}$ are both 1. It follows that $u^{2}=v$ and $v^{2}=u$; we conclude that \newo\ and \newk\ are closed under multiplication. Since complex conjugation interchanges $u$ and $v$, it stabilizes \newo\ and \newk. The norm of $a+bu$ in \newk\ is $(a+bu)(a+bv)=a^{2}-ab+b^{2}$, a non-negative element of \newq. Similarly, we see that $N$ maps the set of non-zero elements of \newo\ into the set of positive integers. The reciprocal of a non-zero element $w$ of \newk\ is $\frac{\bar{w}}{N(w)}$, which lies in \newk, and we see that \newk\ is a field. \newo\ is called ``the ring of Eisenstein integers''.

\begin{theorem}
\label{theoremA.1}\hspace{2em}
The units in \newo\ are $1$, $-1$, $u$, $-u$, $v$ and $-v$.
\end{theorem}

\begin{proof}
Using the multiplicativity of $N$, we see that $a+bu$ is a unit if and only if it has norm 1. Now $4N(a+bu)=4(a^{2}-ab+b^{2})=(2a-b)^{2}+3b^{2}$. So if $a+bu$ is a unit, $b$ is 0, 1, or $-1$. For each such $b$ there are only two possible choices for $a$, making for 6 units in all.
\qed
\end{proof}

\begin{lemma*}
\label{lemmaA.1}
Every $w$ in \newc\ has distance $\le \frac{1}{\sqrt{3}}$ from some $q$ in \newo.
\end{lemma*}

\begin{proof}
Write $w$ as $a+ib\sqrt{3}$ with $a$ and $b$ real. Let $\beta=u-v$. Then $\beta^{2}=u+v-2uv=-3$. So $\beta$ is $i\sqrt{3}$ or  $-i\sqrt{3}$. We are free to modify $w$ by any \newz-linear combination of 1 and $i\sqrt{3}$ and so may assume that $|a| $ and $|b|$ are $\le \frac{1}{2}$. Then $w\bar{w}=a^{2}+3b^{2}$ is $\le 1$, and $w$ lies in the closed unit disc with center 0. Let $H$ be the hexagonal region whose vertices are the units of \newo. If $w$ lies outside of $H$ it evidently has distance $<\frac{2\pi}{12}$ from some unit, but since $\pi<2\sqrt{3}$, this distance is $<\frac{1}{\sqrt{3}}$. If on the other hand, $w$ is in $H$, then it lies in one of the 6 triangular regions whose vertices are 0 and two ``adjacent'' vertices of the hexagonal region. The boundary of each triangular region is an equilateral triangle of edge length 1, and the centroid of such a triangle has distance $\frac{1}{\sqrt{3}}$ from each vertex. It follows easily that $w$ has distance $\le \frac{1}{\sqrt{3}}$ from some vertex of the triangular region, but these vertices lie in \newo. (Alternatively one could draw a tessellation of the plane by equilateral triangles of edge length 1 having vertices in \newo.)
\qed
\end{proof}

\begin{theorem}
\label{theoremA.2}\hspace{2em}
The norm map from \newo\ to the non-negative integers is ``Euclidean''. More precisely, if $l$ and $m$ are in \newo\ with $m\ne 0$, then there are $q$ and $r$ in \newo\ with $l=qm+r$ and $N(r)\le \frac{1}{3}N(m)$.
\end{theorem}

The proof is easy---by the preceding Lemma, $\frac{l}{m}$ has distance $\le \frac{1}{\sqrt{3}}$ from some $q$ in \newo. Let $r$ be $l-qm$, and conclude that $N(r)=N(m)N(\frac{l}{m}-q)\le\frac{1}{3}N(m)$.

\begin{corollary*}
\label{corollaryA.2}\hspace{2em}
Let $l$ and $m$ be elements of \newo\ having no common irreducible factor. Then 1 is an \newo-linear combination of $l$ and $m$.
\end{corollary*}

The argument is by induction on the minimum of $N(l)$ and $N(m)$. If one of $l$ or $m$ is 0, the other has no irreducible factor, must be a unit, and it's clear. If neither is 0 we may assume $0<N(m)\le N(l)$, and we write $l$ as $qm+r$ as in Theorem \ref{theoremA.2}. If an irreducible divides $m$ and $r$, it divides $l$ and $r$, so no such irreducibles exist. Now $N(r)<N(m)$ and the induction assumption tells us that $A(l-qm)+B\cdot m=1$ for some $A$ and $B$ in \newo. So 1 is an \newo-linear combination of $l$ and $m$.

From the corollary one derives all the standard unique factorization facts about \newo. For example, suppose that $m$ in \newo\ is irreducible, and that $l$ in \newo\ represents a non-zero element of $\newo/m$. Then the corollary shows that $1=Al+Bm$ for some $A$ and $B$ in \newo. So $A$ provides a multiplicative inverse to $l$ in  $\newo/m$, and  $\newo/m$ is a field. Since the product of non-zero elements in a field is non-zero we find:

\begin{theorem}
\label{theoremA.3}\hspace{2em}
If an irreducible element, $m$, of \newo\ divides a product, it divides one of the factors.
\end{theorem}

We also see that \newo\ has ``unique factorization into irreducibles'' in the same sense that \newz\ does. I'll make this more precise.  We say that two elements of \newo\ are ``associates'' if each is the product of the other by a unit. From each sextuple of associated irreducibles we choose one, which we deem ``distinguished''. (In the case of \newz, the distinguished irreducible is generally taken to be the positive one.) Using the norm we see easily that every non-zero element of \newo\ is the product of a unit by a (possibly empty) product of powers of distinct distinguished irreducibles.

\begin{theorem}
\label{theoremA.4}\hspace{2em}
Once the distinguished irreducibles have been chosen, the above decomposition is unique.
\end{theorem}

The proof is easy. Suppose one has two distinct such factorizations of $w$. We argue by induction on $N(w)$. If $N(w)=1$, the product of powers in each decomposition is an empty product and the result is clear. If $N(w)>1$, $w$ is divisible by some distinguished irreducible, $m$. By Theorem \ref{theoremA.3}, $m$ appears to positive exponent in both factorizations of $w$ and we apply the induction hypothesis to the two resulting factorizations of $w/m$.

\begin{theorem}
\label{theoremA.5}\hspace{2em}
Suppose $m$ in \newo\ is irreducible. Then:
\begin{enumerate}
\item \label{theoremA.5.a} $m$ divides some prime $p$ of \newz.
\item \label{theoremA.5.b} If $p\equiv 2 \mod{3}$, $m$ is an associate of $p$ and $\newo/m$ has $p^{2}$ elements.
\item \label{theoremA.5.c} If $p\equiv 1 \mod{3}$, $m\bar{m}=p$, $m$ and $\bar{m}$ are not associates, and $\newo/m$ and $\newo/\bar{m}$ each have $p$ elements.
\item \label{theoremA.5.d} If $p=3$, $m$ is an associate of $\beta=u-v$, and $3$ factors as $(-1)\cdot\beta^{2}$.
\end{enumerate}
\end{theorem}

\begin{proofs}
$m$ divides the positive integer $m\bar{m}$. Writing $m\bar{m}$ as a product of powers of primes and applying Theorem \ref{theoremA.3}, we get (\ref{theoremA.5.a}).

Suppose $p\equiv 2 \mod{3}$. If $m$ is not an associate of $p$ its norm must be $p$, and so $m\bar{m}=p$. If $m=a+bu$, then $p=N(m)=(a+b)^{2}-3ab$. We conclude that $p\equiv 0$ or $1\bmod{3}$, contrary to our assumption. If follows that $m$ is an associate of $p$, and then $\newo/m=\newo/p$, which evidently has $p^{2}$ elements.

Suppose $p\equiv 1 \mod{3}$. Then $(\newz/p)^{*}$ is a group whose order is divisible by 3. So $\newz/p$ contains a subgroup $\{1,c,c^{2}\}$ of order 3, with $c$ in \newz, $1<c<p$. The polynomial $x^{3}-1$ factorizes over $\newz/p$ as $(x-1)(x-c)(x-c^{2})$. So in $\newo/p$, $(u-1)(u-c)(u-c^{2})=u^{3}-1=0$, $p$ divides a product in \newo\ without dividing any of the factors, and $\newo/p$ is not a field. We conclude that $p$ is reducible in \newo, and that $p=m\bar{m}$. If $m$ and $\bar{m}$ were associates then $m/\bar{m}$ would be a unit. Replacing $m$ by $um$ or $vm$ if necessary we arrange that the unit is 1 or $-1$. Write $m$ as $au+bv$. Then $b$ is $a$ or $-a$, and $N(m)$ is $a^{2}$ or $3a^{2}$, neither of which is $p$. We conclude that $m$ and $\bar{m}$ are not associates.

Furthermore the integers $0,1,\ldots,p-1$ represent distinct congruence classes in $\newo/m$. Since $x^{3}-1$ factors as $(x-1)(x-c)(x-c^{2})$ over $\newo/m$, and $\newo/m$ is a field, $u\equiv 1$, $c$, or $c^{2}\bmod{m}$.  So any \newz-linear combination of 1, $u$ and $v$ is congruent to an integer mod $m$, and $\newo/m$ cannot have more than $p$ elements. The same holds for $\newo/\bar{m}$. And as $3=(-1)\beta^{2}$ we find that an irreducible dividing 3 can only be an associate of $\beta$.
\qed
\end{proofs}

\renewcommand{\thesection}{\arabic{section}}
\setcounter{section}{0} 
\section{Some results of Euler, Legendre, Kummer and Sylvester, \mbox{together} with analogs} 
\label{section1}

Theorem \ref{theorem1.1} puts no restrictions on the non-zero element $M$ of \newo\ whose representation as a sum of two cubes we're about to study. Though it has its uses, more to the point is the partial converse, Theorem \ref{theorem1.2} (and a few variants of Theorem \ref{theorem1.2} based on the same ``3-descent'' technique). The entire note flows from Theorem \ref{theorem1.2}.

\begin{theorem}
\label{theorem1.1}\hspace{2em}
Let $r$, $s$, $t$ and $M$ be non-zero elements of \newo. If $ur^{3}+vs^{3}+Mt^{3}=0$, then $M=x^{3}+y^{3}$ for some $x$ and $y$ in \newk.
\end{theorem}

\begin{proof}
Take $x$ and $y$ in \newk\ satisfying the linear equations $ux+vy=ur^{3}$, $vx+uy=vs^{3}$. Adding together we get a third equation, $x+y=-(ur^{3}+vs^{3})=Mt^{3}$. Multiplying the three identities we find that $x^{3}+y^{3}=M(rst)^{3}$. Replacing $x$ and $y$ by $\frac{x}{rst}$ and $\frac{y}{rst}$ gives the result.
\qed
\end{proof}

\begin{theorem}
\label{theorem1.2}\hspace{2em}
Let $m$ be an irreducible element of \newo, and $M$ be an associate of $m$ or $m^{2}$. Suppose $x$ and $y$ are in \newk, $x^{3}\ne y^{3}$, and $x^{3}+y^{3}=M$. Then the equation $ur^{3}+vs^{3}+Mt^{3}=0$ has a solution $(r,s,t)$ with $r$, $s$ and $t$ in \newo, $r$ and $s$ prime to $m$ and $rst\ne 0$.
\end{theorem}

\begin{remark*}
Since $1+1=2$, the stipulation $x^{3}\ne y^{3}$ is necessary when $M$ is 2 or $-2$. For then, $ur^{3}+vs^{3}$ is congruent mod $2$ to $u+v=-1$, and cannot be~$-Mt^{3}$.
\end{remark*}

\begin{proof}
Note first that there is a triple $A,B,C$ of distinct elements of \newo\ with $A+B+C=0$, and $ABC$ the product of $M$ by a non-zero cube. (We get such a triple by multiplying $x^{3}$, $y^{3}$ and $-M=-x^{3}-y^{3}$ by a common denominator $D$. For then $ABC=(-xyD)^{3}\cdot M$. Note that $x^{3}$, $y^{3}$ and $-x^{3}-y^{3}$ are distinct.) Choose such a triple with $N(ABC)$ as small as possible.

We'll show that such a triple provides the desired $r$, $s$ and $t$. If $A$, $B$ and $C$ had a common irreducible factor, we could divide them all by that factor. So the minimality assumption tells us they're pairwise coprime, and we may assume $m$ divides $C$. Using unique  factorization in \newo\ and the irreducibility of $m$ we find that $A=ir^{3}$, $B=js^{3}$, $C=kMt^{3}$ where $i$, $j$ and $k$ are units of \newo. Replacing $i$ and $r$ by $-i$ and $-r$ doesn't change $A$. So we may assume $i\in\{1,u,v\}$. Similarly we arrange that $j$ and $k$ are in $\{1,u,v\}$. Now $ABC=(ijk)M(rst)^{3}$ is the product of $M$ by a non-zero cube. Since neither $u$ nor $v$ is a cube in \newk, $ijk=1$. Dividing $A$, $B$ and $C$ by $k$ we arrange that $k=1$. So $ij=1$, and it remains to show that $i\ne j$.

Suppose to the contrary that $i=j=1$ so that $A$, $B$ and $C$ are $r^{3}$, $s^{3}$ and $Mt^{3}$. Since $m$ divides $C$, it divides $A+B=(r+us)(r+vs)(r+s)$, and therefore divides a factor. (When $m=\beta$, the three factors are congruent mod $\beta$, and $\beta$ divides each of them.)  We are free to replace $s$ by $us$ or $vs$, and so may assume $m$ divides $r+s$.  Now let $A^{\prime}$, $B^{\prime}$ and $C^{\prime}$ be $(ur+vs)$, $(vr+us)$ and $r+s$. Then $A^{\prime}+B^{\prime}+C^{\prime}=0$, while $A^{\prime}B^{\prime}C^{\prime}=r^{3}+s^{3}=-Mt^{3}$; the product of $M$ by a non-zero cube.  If $A^{\prime}=B^{\prime}$, $ur+vs=vr+us$, $r=s$ and $A=B$ which is forbidden. If 
$C^{\prime}$ equalled $A^{\prime}$ or $B^{\prime}$ a similar argument would again give the forbidden equality. So $A^{\prime},B^{\prime},C^{\prime}$ is another triple with the desired properties. Now $A^{\prime}B^{\prime}C^{\prime}$ is, as we've seen, $r^{3}+s^{3}=-C$. So $N(ABC)=N(A)N(B)N(C)\ge N(C)=N(A^{\prime}B^{\prime}C^{\prime})$. The minimality assumption tells us that $A$ and $B$ are units. Since $A$ and $B$ are cubes each of them is 1 or $-1$. As $A\ne B$ and $A+B\ne 0$ we have a contradiction.
\qed
\end{proof}

We now combine Theorem \ref{theorem1.2} with a mod $m$ congruence argument to recover (and strengthen) an old result of Sylvester.

\begin{theorem}
\label{theorem1.3}\hspace{2em}
Let $p$ in \newz\ be a prime congruent to $2$ or $5\bmod{9}$, and $M$ be an associate of $p$ or $p^{2}$. If $M\ne 2$ (or $-2$) there are no $x$ and $y$ in \newk\ with $x^{3}+y^{3}=M$. When $M=2$ the only solutions are the obvious nine with $x^{3}=1$ and $y^{3}=1$. The same of course holds when $M=-2$ with $1$ replaced by $-1$. 
\end{theorem}

\begin{proof}
Suppose on the contrary $x^{3}+y^{3}=M$. Since $p$ is $2\bmod{3}$ it is irreducible in \newo, we are in the situation of Theorem \ref{theorem1.2} with $m=p$, and there are $r,s,t$ in \newo\ with $ur^{3}+vs^{3}+Mt^{3}=0$, $rst\ne 0$, and $r$ and $s$ prime to $p$. Then in the field $\newo/p$, $ur^{3}+vs^{3}=0$, and so $-(\frac{r}{s})^{3}=\frac{v}{u}=u$. So $(\newo/p)^{*}$ contains an element of multiplicative order 9. As $p$ is $2$ or $5\bmod{9}$, $p^{2}-1$ is $3$ or $6\bmod{9}$, giving a contradiction.
\qed
\end{proof}

\begin{corollary*}[1 (Pépin, Sylvester, Lucas)]
\label{corollary1.3.1}\hspace{2em}
If $p>2$ is a prime congruent to $2$ or $5\bmod{9}$, neither $p$ nor $p^{2}$ is a sum of two rational cubes.
\end{corollary*}

\begin{corollary*}[2 (Legendre, Euler)]
\label{corollary1.3.2}\hspace{2em}
$x^{3}+y^{3}=4$ has no solutions in \newq. The only solution to $x^{3}+y^{3}=2$ in \newq\ is $(1,1)$.
\end{corollary*}

\begin{remark*}
In contrast there are infinitely many points on the circle $x^{2}+y^{2}=2$ with $x$ and $y$ in \newq. (If one has such a point other than $(1,1)$ then the line joining it to $(1,1)$ has rational slope; conversely when $a$ is the slope of a line (other than the tangent line) through $(1,1)$, and $a$ is rational, that line meets the circle in another rational point.  For example $y=2x-1$ meets the circle at $(-\frac{1}{5}, -\frac{7}{5})$.)  As a consequence there are many triples of integer squares in arithmetic progression---$(\frac{1}{5},\frac{7}{5})$ is on the circle, and $1^{2},5^{2},7^{2}$ are in progression. With cubes it's different.
\end{remark*}

\begin{corollary*}[3 (Legendre)]
\label{corollary1.3.3}\hspace{2em}
Three distinct non-zero cubes cannot lie in arithmetic progression.
\end{corollary*}

For suppose $x^{3}<z^{3}<y^{3}$ are the cubes. By assumption $z\ne 0$.  Then $(\frac{x}{z})^{3}+(\frac{y}{z})^{3}=2$. So $(\frac{x}{z})=(\frac{y}{z})=1$, contradicting distinctness.

\begin{corollary*}[4 (Extending a result of Euler from \newq\ to \newk)]
\label{corollary1.3.4}\hspace{2em}
If $x$ and $y$ are in \newk\ and $y^{2}=x^{3}+1$, then $x^{3}$ is $-1$, $0$ or $8$.
\end{corollary*}

\begin{proof}
There are evidently no $r,s,t$ in \newo\ with $ur^{3}+vs^{3}+2t^{3}=0$ and $rst\ne 0$. Following the descent argument of Theorem \ref{theorem1.2} we conclude that there is no triple of distinct $A,B,C$ in \newo\ with $A+B+C=0$ and $ABC$ the product of 2 by a non-zero cube. It follows that the same holds for $A$, $B$ and $C$ in \newk. Now let $A$, $B$ and $C$ be $1-y$, $1+y$ and $-2$. They sum to zero, and their product is $2(y^{2}-1)=2x^{3}$. So some two of these $A$, $B$ and $C$ are equal, and we get the result by examining the possibilities.  Since $y^{2}$ is 0, 1, or 9, $x^{3}$ is $-1$, 0, or~8.
\qed
\end{proof}

\begin{theorem}
\label{theorem1.4}\hspace{2em}
If $m$ in \newo\ is irreducible with norm a prime congruent to $4$ or $7\bmod{9}$ then no associate $M$ of $m$ or $m^{2}$ is $x^{3}+y^{3}$ with $x$ and $y$ in \newk.
\end{theorem}

\begin{proof}
If $x^{3}+y^{3}=M$, Theorem \ref{theorem1.2} tells us that there are $r,s,t$ in \newo, $rst\ne 0$, with $r$ and $s$ prime to $m$ and $ur^{3}+vs^{3}+Mt^{3}=0$. As in the proof of Theorem \ref{theorem1.3} we find that $ur^{3}+vs^{3}=0$ in the field $\newo/m$, and that $(\newo/m)^{*}$ contains an element of order 9. But this group has order $p-1$, and $p-1$ is $3$ or $6\bmod{9}$.
\qed
\end{proof}

There are a number of simple $M$ we haven't yet discussed---when $m=\beta$ so that $M$ is an associate of $\beta$ or 3, and when $N(m)\equiv 1\mod{9}$. We're also interested in the situation when $M$ is a unit. $\newo/\beta$ is too small for our purposes in these studies, as 1, $u$ and $v$ are all congruent mod $\beta$. Instead we'll work in $\newo/3$, a ring with 9 elements, and we'll even make use of mod $9$ congruences in \newo.  Note that every element, $w$, of \newo\ is congruent mod $3$ to either $0$, $\beta$, $-\beta$ or one of the six units, and that the cube of each unit is 1 or $-1$. So if $w$ in \newo\ is prime to $\beta$, $w$ is the product of a unit by some $1+3t$, and $w^{3} = \pm 1\cdot (1+9t+27t^{2}+27t^{3})$, and so is 1 or $-1\bmod{9}$.

We first deal with the most historically interesting situation, when $M=1$. The technique is akin to that used in the proof of Theorem \ref{theorem1.2}, but there's a variation.

\begin{definition*}
\label{def1.1} 
If $D$ is a non-zero element of \newo, $\ord_{\beta}(D)$ is the largest $n$ such that $\beta^{n}$ divides $D$.
\end{definition*}

\begin{theorem}
\label{theorem1.5}\hspace{2em}
If $A$, $B$ and $C$ are elements of \newo\ summing to $0$, with product a non-zero cube, then $\ord_{\beta}(A)=\ord_{\beta}(B)=\ord_{\beta}(C)$. (Note that $1$, $u$ and $v$ sum to $0$ with product $1$.)
\end{theorem}

\begin{proof}
Suppose there is a counterexample. Choose $A$, $B$ and $C$ with $N(ABC)$ as small as possible. Then we may arrange that $A$, $B$ and $C$ are pairwise coprime, and that $A$ and $B$ are prime to $\beta$. Since we have a counterexample, $\beta$ divides $C$. Since $ABC$ is a cube, 3 divides $C$. Using unique factorization we arrange that $A=ir^{3}$, $B=js^{3}$, $C=kt^{3}$, with $\beta$ dividing $t$, $i$, $j$ and $k$ in $\{1,u,v\}$ and $ijk=1$. Since 3 divides $C$, $A=ir^{3}+js^{3}\equiv 0\mod{3}$, and so either $i+j$ or $i-j$ is divisible by 3. This is impossible unless $i=j$, and so we may arrange that $A=r^{3}$, $B=s^{3}$, $C=t^{3}$ with $\beta$ dividing $t$. We're free to multiply $r$ and $s$ by elements of $\{1,u,v\}$. Furthermore 3 divides $A+B$, so we may assume $r^{3}\equiv 1\mod{3}$ while $s^{3}\equiv -1\mod{3}$. We then arrange that $r\equiv 1\mod{3}$ while $s\equiv -1\mod{3}$. Now consider $ur+vs$, $vr+us$ and $r+s$. Mod 3 they are congruent to $\beta$, $-\beta$ and 0 respectively. So they are all divisible by $\beta$; let $A^{\prime}$, $B^{\prime}$ and $C^{\prime}$ be their quotients by $\beta$. Then $A^{\prime}+B^{\prime}+C^{\prime}=0$, while $A^{\prime}B^{\prime}C^{\prime}=(\frac{1}{\beta^{3}}\cdot(r^{3}+s^{3})=\frac{-t^{3}}{\beta^{3}}=\frac{-C}{\beta^{3}}$. Furthermore, $\ord_{\beta}(A^{\prime})=\ord_{\beta}(B^{\prime})=0$ while $\ord_{\beta}(C^{\prime})>0$. So $A^{\prime},B^{\prime},C^{\prime}$ is once more a counterexample, and $N(ABC)\ge N(C)=27\cdot N(A^{\prime}B^{\prime}C^{\prime})$, contradicting the minimality assumption.
\qed
\end{proof}

\begin{corollary*}[1 (Kummer)]
\label{corollary1.5.1}\hspace{2em}
If $x,y,z$ are non-zero elements of \newo, then $x^{3}+y^{3}+z^{3}\ne 0$.
\end{corollary*}

\begin{proof}
We may assume $x$, $y$ and $z$ are pairwise coprime. If they are all prime to $\beta$, then, as we've observed, each of $x^{3}$, $y^{3}$ and $z^{3}$ is 1 or $-1\bmod{9}$, so their sum isn't divisible by 9.  So we may assume $\beta$ divides $z$. We now let $A$, $B$ and $C$ be $x^{3}$, $y^{3}$ and $z^{3}$, and invoke Theorem \ref{theorem1.5}.
\qed
\end{proof}

As an immediate consequence we find

\begin{corollary*}[2]
\label{corollary1.5.2}\hspace{2em}

\begin{enumerate}
\item \label{corollary1.5.2a} If $x$ and $y$ are elements of \newk\ with $x^{3}+y^{3}=1$ then either $x=0$, so $y\in\{1,u,v\}$ or $y=0$, so $x\in\{1,u,v\}$.
\item \label{corollary1.5.2b} (Claimed by Fermat, proved (more or less) by Euler.) The only points on the plane curve $x^{3}+y^{3}=1$ with rational $x$ and $y$ are $(1,0)$ and $(0,1)$.
\end{enumerate}
\end{corollary*}

\begin{remark*}
(\ref{corollary1.5.2b}) above appears as Exercise 3.17(ii) of \cite{2}. I think this exercise is meant as a joke.  Al says: ``Ken was a master of this, a consummate mathematician who loved to tell stories and make jokes\ldots . My experience is that this kind of playfullness draws people into the culture''.
\end{remark*}

I wouldn't expect a proof of FLT(3) from a new entrant into the mathematical world, but I hope that after reading the relevant sections of \cite{2}, the entrant would be prepared to look at this note and follow the proofs of Theorem \ref{theorem1.5} and the above corollaries.

We can also strengthen Theorem \ref{theorem1.5} slightly.

\begin{corollary*}[3]
\label{corollary1.5.3}\hspace{2em}
If $A$, $B$ and $C$ are elements of \newo\ summing to 0, with product a non-zero cube, then $A$, $B$ and $C$ are in some order the products of a non-zero element of \newo\ by 1, $u$ and $v$.
\end{corollary*}

\begin{proof}
Let $F$ and $G$ be the elements of \newk\ satisfying the linear equations $uF+vG=A$, $vF+uG=B$. Then $F+G=C$; multiplying these identities we see that $F^{3}+G^{3}=ABC$ which is $D^{3}$ for some non-zero $D$. So $(\frac{F}{D})^{3}+(\frac{G}{D})^{3}=1$. By Corollary 2(\ref{corollary1.5.2a}), either $F=0$ or $G=0$. In the first case, $A$, $B$ and $C$ are $vG$, $uG$ and $G$. In the second they are $uF$, $vF$ and $F$.
\qed
\end{proof}

\begin{theorem}
\label{theorem1.6}\hspace{2em}
A unit of \newo\ other than $1$ or $-1$ is not a sum of two cubes in \newk.
\end{theorem}

\begin{proof}
We may assume the unit is $u$. If the result is false we produce in the usual way a triple $A,B,C$ in \newo\ with $A+B+C=0$ and $ABC=u\cdot(\mbox{a non-zero cube})$, and we select such a triple with $N(ABC)$ minimal. Then $A$, $B$ and $C$ are pairwise prime, and we find that $A=ir^{3}$, $B=js^{3}$, $C=kt^{3}$ with $i$, $j$ and $k$ in $\{1,u,v\}$. Since $ABC$ is the product of $u$ and a non-zero cube, $ijk=u$ and some two of $i$, $j$ and $k$ are equal. We may assume $i=j=1$. Then $k=u$, and $A=r^{3}$, $B=s^{3}$, $C=ut^{3}$.

Now define $A^{\prime}$, $B^{\prime}$ and $C^{\prime}$ as in the proof of Theorem \ref{theorem1.2}. They're another triple summing to 0, and their product is $r^{3}+s^{3}=-ut^{3}=-C$. Using the minimality property of the triple $A,B,C$ we find that $A$ and $B$ are units. Since $A$ and $B$ are cubes and $A+B\ne 0$, $A=B$ and so $ABC=2\cdot(\mbox{a cube})$ contradicting our assumptions.
\qed
\end{proof}

\begin{theorem}
\label{theorem1.7}\hspace{2em}
Let $M$ be an associate of $\beta$ or $\beta^{2}$ other than $\beta$ (or $-\beta$). Then $x^{3}+y^{3}=M$ has no solutions in \newk.
\end{theorem}

\begin{proof}
If $x^{3}+y^{3}=M$, Theorem \ref{theorem1.2} provides an $r,s,t$ in \newo\ with $ur^{3}+vs^{3}+Mt^{3}=0$, and $r$ and $s$ prime to $\beta$. Suppose first that $M$ is an associate of $\beta^{2}=-3$. Then, mod $3$, $ur^{3}+vs^{3}$ is $\pm u\pm v$, and so cannot, as $Mt^{3}$ is, be a multiple of $3$, giving a contradiction. Suppose next that $M=u\beta=u(u-v)=v-1=u+2v$. Then, mod $9$, $ur^{3}+vs^{3}$ is $\pm u\pm v$, while $Mt^{3}$ is $(u+2v)$, $(-u-2v)$ or $0$, giving a contradiction. We argue similarly when $M$ is $-u\beta$, $v\beta$ or $-v\beta$.
\qed
\end{proof}

\begin{corollary*}[(Legendre)]
\label{corollary1.7.legendre}\hspace{2em}
$x^{3}+y^{3}=3$ has no solution in \newq.
\end{corollary*}

\begin{remark*}
In contrast, $x^{3}+y^{3}=9$ has infinitely many solutions in \newq.  (See Exercise 3.18 Take It Further, on pages 78--79 of \cite{2} for a sketch of the proof.)  Since $9=(\beta^{3})(\beta)$, this means that $\beta$ and $-\beta$ can be written in infinitely many ways as the sum of two cubes in \newk.
\end{remark*}

\section{More complicated descents}
\label{section2}

We've been proving non-existence of solutions (other than the trivial ones) to $x^{3}+y^{3}=M$, $x$ and $y$ in \newk, provided the irreducible factors of $M$ form a single class of associates. (Of course restrictions have to be made on $M$ for this to hold. For example, 17, which is irreducible in \newo, is $(\frac{18}{7})^{3}-(\frac{1}{7})^{3}$, and at the end of the note we'll see that although $m=5u+2v$, which is irreducible of norm 19, isn't such a sum, both $um$ and $vm$ are.)

We proceed to more complicated $M$. First we'll get a result when $M$ is $\beta p$ or $\beta p^{2}$ with $p$ a prime $\equiv 2$ or $5\bmod{9}$. This recovers a theorem of Sylvester---for such $p$ neither $9p$ nor $9p^{2}$ is a sum of two cubes in \newq.

Next consider primes $p\equiv 4$ or $7\bmod{9}$. Such $p$ factor in \newo\ as $\pi\bar{\pi}$ with $\pi\equiv 1\mod{3}$ and we take $M$ to be an associate of $p$ or $p^{2}$ not lying in \newz. We'll show that $M$ is not $x^{3}+y^{3}$ with $x$ and $y$ in \newk. (In the proof we'll impose a supplementary condition on $p$ involving $\pi$ and $\bar{\pi}$. But this condition is in fact \emph{always} satisfied, and for any given $p$ it's easily checked.) Of course we may assume $M$ is $pu$ or $p^{2}u$.

Finally suppose $p\equiv 1\mod{3}$. Then under certain circumstances we'll show that neither $3p$ nor $3p^{2}$ is $x^{3}+y^{3}$ with $x$ and $y$ in \newq. The $p< 100$ for which the circumstances don't apply are 61, 67 and 73. For each of these $p$ we'll show that $3p$ \emph{is} such a sum, but we don't know about $3p^{2}$, or what happens for many larger $p$.

\begin{theorem}
\label{theorem2.1}\hspace{2em}
Suppose $p$ in \newz\ is prime and $2$ or $5\bmod{9}$. (Note that $p$ is irreducible in \newo.)  Then neither $\beta p$ nor $\beta p^{2}$ is $x^{3}+y^{3}$ with $x$ and $y$ in \newk.
\end{theorem}

\begin{proof}
We'll deal with $\beta p$; the case of $\beta p^{2}$ is essentially the same. If $x^{3}+y^{3}=\beta p$, we, in the usual obvious fashion, construct $A$, $B$ and $C$ in \newo\ with $A+B+C=0$, and $ABC$ the product of $\beta p$ by a non-zero cube. We choose a triple of this sort with $N(ABC)$ minimal, and find that $A$, $B$ and $C$ are pairwise prime.  $\beta$ divides $ABC$, so we may assume it divides $C$. Using unique factorization, the irreducibility of $\beta$ and $p$, and the simple arguments of Theorem \ref{theorem1.2} we find that there are $i$, $j$ and $k$ in $\{1,u,v\}$ and $r,s,t$ in \newo\ such that $ijk=1$ and one of the following holds:

\begin{enumerate}
\item \label{proof2.1a} \parbox{1in}{$A=ir^{3}$}\parbox{1in}{$B=jps^{3}$}\parbox{1in}{$C=k\beta t^{3}$}
\item \label{proof2.1b} \parbox{1in}{$A=ir^{3}$}\parbox{1in}{$B=js^{3}$}\parbox{1in}{$C=k\beta pt^{3}$}
\end{enumerate}

We're free to multiply $i$, $j$ and $k$ by a single element of $\{1,u,v\}$, ensuring that $k=1$ and $ij=1$. Note that $C\equiv 0$, $\beta$ or $-\beta\bmod{9}$ in case (\ref{proof2.1a}). We first dispose of case (\ref{proof2.1a}). If $i=j=1$, then $A+B\equiv \pm 1\pm 2$ or $\pm 1\pm 5\bmod{9}$ and so can't be congruent to $-C$. If $i=u$, $j=v$, then $A+B\equiv \pm u\pm 2v$ or $\pm u\pm 5v\bmod{9}$, while $\beta = u-v$ and we get a contradiction. Finally when $i=v$, $j=u$, then $A+B\equiv \pm v\pm 2u$ or $\pm v\pm 5u$, and so $A+B+C$ can't be $0\bmod{9}$, and we must be in case (\ref{proof2.1b}).

But now we switch to arguing mod ${p}$. Mod $p$, $0\equiv A+B=ir^{3}+js^{3}$, and it follows that $i/j$ is a cube in the multiplicative group of $\newo/p$, a group with $p^{2}-1$ elements. But $p^{2}-1\equiv 3$ or $6\bmod{9}$, so $(\newo/p)^{*}$ contains no element of order 9, $i/j$ cannot be $u$ or $v$, and so $i=j=k=1$ and $A$, $B$, $C$ are $r^{3}$, $s^{3}$ and $\beta pt^{3}$ We now finish things off as in the proof of Theorem \ref{theorem1.2}, arranging first that $p$ divides $r+s$, then defining $A^{\prime}$, $B^{\prime}$ and $C^{\prime}$ in the usual fashion, and finally using the minimality assumption to get a contradiction.
\qed
\end{proof}

\begin{remark*}
The above result differs in kind from Theorem \ref{theorem1.3}, where it was shown that no associate of $p$ or $p^{2}$ is a sum of two cubes with the single exceptions $2=1+1$, $-2=(-1)+(-1)$. For although $18=(2\beta)\beta^{3}$ is not a sum of two cubes in \newk, by Theorem \ref{theorem2.1}, the same doesn't hold for its associate $18u$. Note that $(1-2v)^{3}=1-6v+12u-8=19u+v$, and so $18u=(19u+v)+1=(1-2v)^{3}+1^{3}$; the classical ``tangent line'' and ``secant line'' constructions, generalized from \newq\ and the reals to \newk\ and the complexes, show that the equation $x^{3}+y^{3}=18u$ has infinitely many solutions in \newk.

Since, as we've noted, $9=\beta^{4}$ is the product of $\beta$ and a cube, Theorem \ref{theorem2.1} gives us:
\end{remark*}

\begin{corollary*}[]
\label{corollary2.1}\hspace{2em}
If $p$ is as in Theorem \ref{theorem2.1}, neither $9p$ nor $9p^{2}$ is a sum of two cubes in \newk. In particular neither $9p$ nor $9p^{2}$ is $x^{3}+y^{3}$ with $x,y$ in \newq---this result is due to Sylvester.
\end{corollary*}

We'll next be concerned with associates of $p$ and $p^{2}$ where the prime $p$ is $1\bmod{3}$ and so can be written as $\pi\bar{\pi}$, where $\pi$ and $\bar{\pi}$ are not associates. Since every mod 3 congruence class in \newo\ containing elements prime to $\beta$ is represented by a unit, we are free to assume $\pi\equiv 1\mod{3}$. If we write $\pi$ as $au+bv$, this means that $a\equiv b\equiv -1\mod{3}$.

When $p$ is 4 or $7\bmod{9}$ the results for $p$ and $p^{2}$ are now known---it was conjectured by Sylvester that $x^{3}+y^{3}=p$ and $x^{3}+y^{3}=p^{2}$ each have infinitely many solutions with $x$ and $y$ in \newq. This is now established; see \cite{1} for a discussion. Elkies announced a proof but has not published full details. Dasgupta and Voight published details of a related argument, but had to impose a supplementary condition on $p$. For $p\equiv 4$ or $7\bmod{9}$ see Elkies' mind-boggling tables in the bibliography of \cite{1} for the solutions of $x^{3}+y^{3}=p$ and $x^{3}+y^{3}=p^{2}$, with $x$ and $y$ in \newq\ having the smallest possible common denominator.

The situation for associates $M$ of $p$ or $p^{2}$ that do not lie in \newz\ is quite different. We may of course assume that $M$ is $up$ or $up^{2}$, and we show that these $M$ (when $p$ is $4$ or $7\bmod{9}$) cannot be written as $x^{3}+y^{3}$ with $x$ and $y$ in \newk. There's a restriction on $p$ (though it turns out to be only apparent) that we'll first discuss.

\begin{definition}
\label{def2.1} 
For $p\equiv 1\mod{3}$ write $p$ as $\pi\bar{\pi}$ with $\pi\equiv 1\mod{3}$. $p$ satisfies \textup{(I)} if $\bar{\pi}$ is a cube in the multiplicative group of the field $\newo/\pi$.
\end{definition}

Write $\pi$ as $au+bv$, with $a\equiv b\equiv -1\mod{3}$. Note that \textup{(I)} is satisfied precisely when the integer $\bar{\pi}+\pi$ is a cube in $(\newo/\pi)^{*}$. But $\bar{\pi}+\pi=-a-b$, and to establish \textup{(I)} for a given $p$ it's enough to show that $a+b$ is a cube in $(\newz/p)^{*}$. Here's a table for $p\le 73$.

\hspace{.5in}
\begin{tabular}{ll}
$p=7$ & $a+b = 2+(-1) = 1$\\
$p=13$ & $a+b = -4+(-1) = -5$\\
$p=19$ & $a+b = 5+2 = 7$\\
$p=31$ & $a+b = 5+(-1) = 4$\\
$p=37$ & $a+b = -7+(-4) = -11$\\
$p=43$ & $a+b = -7+(-1) = -8$\\
$p=61$ & $a+b = -4+5 = 1$\\
$p=67$ & $a+b = -7+2 = -5$\\
$p=73$ & $a+b = -1+8 = 7$\\
\end{tabular}

\begin{observation*}[]
\label{observation2.1} 
For all $p\equiv 1\mod{3},\ p\le 73$, \textup{(I)} is satisfied.
\end{observation*}

This is obvious when $p$ is 7, 43, or 61. Now consider the other cases:

\hspace{.5in}
\begin{tabular}{ll}
$p=13$ & $-5\equiv 8$\\
$p=19$ & $7\equiv 64$\\
$p=31$ & $4\equiv -27$\\
$p=37$ & since $4^{3}\equiv -10,\ (16)^{3}\equiv 100 \equiv -11$\\
$p=67$ & In $(\newz/p)^{*},\ -5 = \frac{-40}{8} = \frac{27}{8}$\\
$p=73$ & In $(\newz/p)^{*},\ 5^{3}\equiv -21,\ 6^{3}\equiv -3,$ and so $7 = \frac{-21}{-3} \equiv (\frac{5}{6})^{3}$\\
\end{tabular}

\begin{lemma*}
Suppose $p\equiv 4$ or $7\bmod{9}$ satisfies \textup{(I)}. Suppose furthermore that $A$, $B$ and $C$ are pairwise prime elements of \newo\ with $A+B+C=0$ and $ABC$ the product of $up$ or $up^{2}$ by a non-zero cube. Then two of $A$, $B$ and $C$ must be the product of a unit and a cube.
\end{lemma*}

\begin{proof}
We'll treat the case of $up$---that of $up^{2}$ is essentially the same. Let $\pi$ and $\bar{\pi}$ be as above. If the lemma fails, unique factorization and the irreducibility of $\pi$ and $\bar{\pi}$ show that we may arrange that $A=i\pi r^{3}$, $B=j\bar{\pi}s^{3}$, $C=kt^{3}$ with $i$, $j$ and $k$ in $\{1,u,v\}$. Since \textup{(I)} holds for $\pi$, we find that mod $\pi$, $B+C\equiv j\cdot (\mbox{a cube})+k\cdot (\mbox{a cube})$. So $j/k$, which lies in $\{1,u,v\}$, represents a cube in the group $(\newo/\pi)^{*}$ of order $p-1$. If $j/k$ is $u$ or $v$, $(\newo/\pi)^{*}$ contains an element of order 9; since $p-1\equiv 3$ or $6\mod{9}$ this is impossible. So $j=k$. Similarly, arguing mod $\bar{\pi}$, we find that $i=k$. Then $ijk=k^{3}=1$, and so $ABC=(ijk)\cdot p\cdot (rst)^{3}$, contradicting the fact that $ABC$ is the product of $up$ and a cube.
\qed
\end{proof}

\begin{theorem}
\label{theorem2.2}\hspace{2em}
Let $p$ be as in the Lemma. If $p$ satisfies \textup{(I)}, then neither $up$ nor $up^{2}$ is $x^{3}+y^{3}$ with $x$ and $y$ in \newk. (So, by the observation, for all $p\le 73$ and $\equiv 4$ or $7\bmod{9}$, neither $up$ nor $up^{2}$ admits such a representation.)
\end{theorem}

\begin{proof}
We'll deal with $up$; the argument for $up^{2}$ is the same. Let $\pi$ and $\bar{\pi}$ be as above. If the result fails, and $up=x^{3}+y^{3}$, we construct in the usual obvious fashion $A$, $B$ and $C$ in \newo\ with $A+B+C=0$, and $ABC$ the product of $up$ by a non-zero cube, and choose such a triple with $N(ABC)$ minimal. Then $A$, $B$ and $C$ are pairwise prime. Using \textup{(I)} and the previous lemma we find that $p$ must divide one of $A$, $B$ and $C$, and we may assume $p$ divides $C$. 
Standard simple arguments show that $A$, $B$ and $C$ are $ir^{3}$, $js^{3}$ and $k(up)t^{3}$ with $i$, $j$ and $k$ in $\{1,u,v\}$ and $ijk = 1$. Then, in $\newo/\pi$, $ir^{3}+js^{3}\equiv 0$, and so $i/j$ represents a cube in the group $(\newo/\pi)^{*}$ of order $p-1$. Since $p-1\equiv 3$ or $6\mod{9}$, $i/j$ cannot be $u$ or $v$. Then $i=j=k$ and $A,B,C$ are $r^{3}$, $s^{3}$ and $(up)t^{3}$. As in the proof of Theorem \ref{theorem1.2} we arrange that $\pi$ divides $r+s$, and we define $A^{\prime}$, $B^{\prime}$ and $C^{\prime}$ as in the proof of that theorem. This gives a new counterexample and the minimality assumption gives the desired contradiction.
\qed
\end{proof}

\begin{remark*}
As I suggested earlier, \textup{(I)} indeed holds for all $p\equiv 1\mod{3}$. This is a beautiful result of Eisenstein and/or Gauss, connected with ``cubic reciprocity''. In any case, for any given $p$ it's easy to verify---since $(\newz/p)^{*}$ is a cyclic group of order $p-1$ one only needs show that $(a+b)^{\frac{p-1}{3}}$ is congruent to $1\bmod{p}$; you could extend my examples far past 73 using a computer. Cubic reciprocity has found a place in some texts accessible to undergraduates, notably that of Ireland and Rosen, but it requires more background than that of this note.
\end{remark*}

I'll turn now to a more complex problem that can be stated without any reference to \newo\ or \newk. Let $p$ be a prime in \newz, congruent to $1\bmod{3}$, and $M$ be $3p$ or $3p^{2}$. When is $M=x^{3}+y^{3}$ with $x$ and $y$ in \newq?

The story of my attack on this is good, and I'll give a polished version here---all the events I relate occurred, though not in the order I depict.

A glance at the \emph{Online Encyclopedia of Integer Sequences} showed there were no such representations of $3p$ for $p\le 37$. A more careful study of the $L$-functions and modular forms database showed there were no representations of $3p$ or $3p^{2}$ for $p<61$. (Care must be exercised---if you search under \newk\ rather than under \newq\ the ``curves'' you're looking for won't appear, even when $p=7$. And searching under \newq\ you'll need to know something about Weierstrass normal forms of elliptic curves.) After making the search I realized however that there were representations of $3p$ for $p=61$, 67, and 73. Suppose first that $p=61$, and consider the triple $A,B,C$ of 64, $-3$ and $-61$. These sum to 0 and $ABC$ is $3\cdot 61\cdot (\mbox{a cube})$. This guarantees that the representation exists. One way of seeing this is to choose $F$ and $G$ in \newk\ with $uF+vG=A$, $vF+uG=B$. We argue as in the proof of Theorem \ref{theorem1.1}. $F+G=C$, and multiplying the identities together we find that $F^{3}+G^{3}=ABC = 3(61)\cdot (\mbox{a cube})$. Of course $F$ and $G$ are in \newk, not \newq, but this is not hard to remedy.

Another more computer-friendly proof is to use an identity of Lucas. Now $A$ and $B$ are indeterminates and we take $x$ to be $A^{3}-B^{3}+6A^{2}B+3AB^{2}$ and $y$ to be $B^{3}-A^{3}+3A^{2}B+6AB^{2}$. Then $x+y=9AB(A+B)$ while $x^{2}-xy+y^{2}=3(A^{2}+AB+B^{2})^{3}$. If we set $C$ equal to $-A-B$ and multiply these two identities we find that $x^{3}+y^{3}=ABC\cdot (\mbox{a cube})$. In particular taking $A$ and $B$ to be 64 and $-3$ we find that $x=190171$, $y=-295579$, and $x^{3}+y^{3}=(-183)(46956)^{3}$. A similar approach works when $p=67$. Now $A$, $B$ and $C$ are 64, 3, and $-67$, and $ABC=(-4)^{3}\cdot 3p$. When $p=73$, take $A$, $B$, $C$ to be 81, $-8$ and $-73$, so that $ABC=6^{3}\cdot 3p$, and once again get a representation.

The sequence $7,13,19,31,37,43,61,67,73$, with the last three entries exceptional, aroused some memories in me and I was led to consider exceptional circumstances in which $3p$ or $3p^{2}$ might be a sum of two rational cubes. Two circumstances presented themselves.

\begin{definition}
\label{def2.2}
$p$ is Exceptional $A$ if it has an irreducible factor $\pi$ in \newo\ with $\pi\equiv \mbox{an integer}\bmod{9}$.
\end{definition}

Such a $\pi$ is evidently congruent to 1 or $-1\bmod{3}$. So if we write $\pi$ as $au+bv$, the condition amounts to saying that 9 divides $a-b$ for this particular $\pi$. In particular, looking at the tables following Definition \ref{def2.1} we find that 61, 67 and 73 are Exceptional $A$, while the $p<61$ are not. Also, when $p=79$ we take $\pi$ to be $-10u-7v$, while when $p=97$ we take $\pi$ to be $11u+8v$, concluding that 79 and 97 are not Exceptional $A$.

\begin{definition}
\label{def2.3}
$p$ is Exceptional $B$ if $3$ is a cube in $\newz/p$.
\end{definition}

Now $(\newz/p)^{*}$ is a cyclic group of order divisible by 3. It follows that $p$ is Exceptional $B$ precisely when $3^{\frac{p-1}{3}}\equiv 1\bmod{p}$. Observe that 61, 67 and 73 are Exceptional $B$.  For:

\begin{tabular}{ll}
When $p=61$, & $3^{5}\equiv -1$ and $3^{20}\equiv 1$\\
When $p=67$, & $3^{4}\equiv 14$, $3^{8}\equiv 196\equiv -5$, $3^{12}\equiv -70\equiv -3$  and so $3^{11}\equiv -1$\\
When $p=73$, & $3^{4}\equiv 8$, $3^{8}\equiv 64\equiv -9$, $3^{12}\equiv 1$\\
\end{tabular}

(Or just look at the cubes of $4$, $-4$, and $-6\bmod{61}$, $67$ and $73$.) One checks quickly that the other $p\equiv 1\mod{3}$ and $<100$ are not Exceptional $B$. For example when $p=79$, $3^{4}\equiv 2$, $3^{12}\equiv 8$, and $3^{13}\equiv 24$, and so is not 1 or $-1\bmod{79}$.

There are other ways of phrasing the Exceptional $A$ condition that make it easier to verify. Suppose $p$ is Exceptional $A$ and $\pi$ is an irreducible of norm $p$ congruent to an integer mod 9. Write $\pi$ as a \newz-linear combination of 1 and $u$. Then the coefficient of $u$ must be divisible by 9. So $\pi=x+(9y)(u)$, and $p=N(\pi)=x^{2}-9xy+81y^{2}$. In other words $p$ is Exceptional $A$ precisely when it is represented by the quadratic form $x^{2}-9xy+81y^{2}$. Since $4x^{2}-36xy+324y^{2}=(2x-9y)^{2}+243y^{2}$, one finds that $p$ is Exceptional $A$ precisely when $4p$ can be written as $x^{2}+243y^{2}$ for some $x$ and $y$. Also since $244=1^{2}+243$, $268=5^{2}+243$, $292=7^{2}+243$, $412=13^{2}+243$, $604=19^{2}+243$, $772=23^{2}+243$, we find that 61, 67, 73, 103, 151 and 193 are Exceptional $A$, while the other $p\equiv 1\mod{3}$ and $<200$ are not.

Before going on to the statement and proof of Theorem \ref{theorem2.3}, I'll show that the primes 79 and 97, neither of which is Exceptional $A$ or Exceptional $B$ do abide by the Eisenstein and/or Gauss result, and satisfy \textup{(I)} of Definition \ref{def2.1}.  When $p=79$, $\pi=-10u-7v$ and we need to show that 17 is a cube in $\newz/79$. When $p=97$, $\pi=11u+8v$ and we need to show that 19 is a cube in $\newz/97$. Now when $p=73$, $17^{2}\equiv -3$, $17^{8}\equiv 81\equiv 8$ and $17^{24}\equiv 512 \equiv 1$. And when $p=97$, $19^{2}\equiv -27$, $19^{4}\equiv 729\equiv 50 \equiv \frac{3}{2}$, and so $19^{16}\equiv \frac{81}{16}\equiv -1$, $19^{32}\equiv 1$, and in both cases \textup{(I)} is satisfied.

\begin{theorem}
\label{theorem2.3}\hspace{2em}
Let $p$ be a prime $\equiv 1\mod{3}$ which satisfies \textup{(I)} of Definition \ref{def2.1}. Suppose furthermore that $p$ is neither Exceptional $A$ nor Exceptional $B$. Then neither $3p$ nor $3p^{2}$ is a sum of two cubes in \newk. (In view of the preceding paragraph the conclusion holds for all $p<100$ other than $61$, $67$ and $73$.)
\end{theorem}

\begin{proof}
We'll treat the case of $3p$, that of $3p^{2}$ being much the same. Suppose on the contrary we have a solution in \newk\ of $x^{3}+y^{3}=3p$. In the accustomed fashion we find $A$, $B$ and $C$ in \newo\ with $A+B+C=0$, and $ABC$ the product of $3p$ by a non-zero cube, and we choose such a triple with $N(ABC)$ minimal. Then $A$, $B$ and $C$ are pairwise prime. Write $p$ as $\pi\bar{\pi}$ with $\pi\equiv 1\mod{3}$, so that $3p$ is the product of $-\beta^{2}$, $\pi$ and $\bar{\pi}$.  Since $\beta$ divides $ABC$, we may assume $\beta$ divides $C$. Now factorize $A$, $B$ and $C$ into powers of products of irreducibles. A standard argument shows that there are $i,j,k$ in $\{1,u,v\}$ with $ijk=1$, and $r,s,t$ in \newo\ such that $A=i\cdot (\ )r^{3}$, $B=j\cdot (\ )s^{3}$, $C=k\cdot (\ )3t^{3}$, where each term in parentheses is either $1$, $\pi$, $\bar{\pi}$ or $p$, and the product of these three parenthesized terms is $p$. Then mod $3$, $A+B$ is $\pm i\pm j$. If $i\ne j$, then neither $i+j$ nor $i-j$ is divisible by $3$. So $i=j$. It follows that $i=j=k$ and that we can assume that they are all $1$. We're free to interchange $\pi$ and $\bar{\pi}$, and interchange $A$ and $B$, and so are presented with just four possibilities.

\begin{enumerate}
\item \label{proof2.3a} \parbox{1in}{$A=\pi r^{3}$}\parbox{1in}{$B=\bar{\pi} s^{3}$}\parbox{1in}{$C=3t^{3}$}
\item \label{proof2.3b} \parbox{1in}{$A=p r^{3}$}\parbox{1in}{$B=s^{3}$}\parbox{1in}{$C=3t^{3}$}
\item \label{proof2.c} \parbox{1in}{$A=\pi r^{3}$}\parbox{1in}{$B=s^{3}$}\parbox{1in}{$C=3\bar{\pi} t^{3}$}
\item \label{proof2.3d} \parbox{1in}{$A=r^{3}$}\parbox{1in}{$B=s^{3}$}\parbox{1in}{$C=3pt^{3}$}
\end{enumerate}

We'll use our assumption to eliminate (\ref{proof2.3a}), (\ref{proof2.3b}) and (\ref{proof2.c}). If (\ref{proof2.3a}) holds then mod~$3$, $A+B$ is $\pm\pi\pm\bar{\pi}$. For $D$ in \newo, $D\ne 0$, let $\ord_{\beta}(D)$ be the exponent to which $\beta$ appears in the factorization of $D$. Since $\pi\equiv\bar{\pi}\equiv 1\bmod{3}$, $\ord_{\beta}(\pi +\bar{\pi})=0$. If we write $\pi$ as $au+bv$, then $\pi -\bar{\pi}=\beta(a-b)$. Since $p$ is not Exceptional $A$, $a-b$ is $3$ or $6\bmod{9}$, and $\ord_{\beta}(\pi -\bar{\pi})=\ord_{\beta}(\beta)+\ord_{\beta}(a-b)=3$, but since $C=3t^{3}$, $\ord_{\beta}(A+B)$ is $\equiv 2\bmod{3}$, and cannot be $0$ or $3$, giving a contradiction.

Next suppose we're in situation (\ref{proof2.3b}). Then $s^{3}+3t^{3}\equiv 0\mod{\pi}$. So $3$ is a cube in $\newo/\pi$, and therefore a cube in $\newz/p$, a possibility foreclosed since $p$ is not Exceptional $B$. In situation (\ref{proof2.c}), $s^{3}+3\bar{\pi}t^{3}\equiv 0\mod{\pi}$. Since $p$ satisfies \textup{(I)}, $\bar{\pi}$ is a cube in $\newo/\pi$, and it follows that $3$ is a cube in $\newo/\pi$, and that $p$ is Exceptional $B$, an excluded possibility.  So we're in situation (\ref{proof2.3d}). Now we proceed as in the proof of Theorem (\ref{theorem1.2}), first arranging that $p$ divides $r+s$, then defining a new counterexample $A^{\prime},B^{\prime},C^{\prime}$, and then using the minimality assumption on $N(ABC)$ to show that $A$ and $B$ are units, giving a contradiction, and completing the proof.
\qed
\end{proof}

\begin{remark*}
The reader will surely have guessed by now that $p$ is Exceptional $A$ if and only if it is Exceptional $B$. This is in fact a consequence of cubic reciprocity, and was proved by Eisenstein and/or Gauss, but the proof seems to be of the same level of difficulty as the proof that \textup{(I)} of Definition \ref{def2.1} holds for all $p\equiv 1\mod{3}$. But if we grant these cubic reciprocity results we have:
\end{remark*}

\begin{corollary*}[]
\label{corollary2.3.1}\hspace{2em}
Let $p$ be a prime $\equiv 1\mod{3}$ such that $4p$ cannot be written as $a^{2}+243b^{2}$ with $a$ and $b$ in \newz. Then neither $3p$ nor $3p^{2}$ is $x^{3}+y^{3}$ with $x$ and $y$ in \newq.
\end{corollary*}

(This is a less well-known result of Sylvester and Pépin---for a while I imagined I was the first to discover it.) I've been informed on MathOverflow that the converse to the Corollary fails. Although 307 is Exceptional A and B, 921 is not a sum of two rational cubes. The proof is difficult, using a lot of algebraic number theory, and my informant relied on a computer program designed to treat just such questions.

In these notes we've given a complete proof that the conclusion of the corollary holds for all $p<100$ (apart from $61$, $67$ and $73$ which do not satisfy the restriction). For larger $p$ for which $4p\ne a^{2}+243b^{2}$ our proof is not complete. To complete it one needs to check that \textup{(I)} holds for the given $p$ and that $3$ is not a cube in $\newz/p$ for the given $p$, but this can be verified without much computation.

We present one more result that encapsulates what we know when $M$ is an associate of $\pi$ or $\pi^{2}$, and $\pi$ is irreducible of prime norm $p$, $p\equiv 1\mod{9}$.

\begin{theorem}
\label{theorem2.4}\hspace{2em}
Suppose $p\equiv 1\mod{9}$ is not Exceptional $A$. Write $p$ as $\pi\bar{\pi}$ with $\pi\equiv 1\mod{3}$. Then neither $\pi$ nor $\pi^{2}$ is a sum of two cubes in \newk.
\end{theorem}

\begin{proof}
I'll argue with $\pi$, the case of $\pi^{2}$ being much the same. If the theorem fails, Theorem \ref{theorem1.2} shows there are coprime non-zero $r$, $s$ and $t$ in \newo\ with $ur^{3}+vs^{3}+\pi t^{3}=0$.  Suppose first that $\beta$ divides $r$. Then, mod $3$, $ur^{3}\equiv 0$, $vs^{3}\equiv \pm v$ and $\pi t^{3}\equiv \pm 1$, giving a contradiction. So we may assume that $r$ (and similarly $s$) is prime to $\beta$. But then, mod $9$, $ur^{3}+vs^{3}\equiv 1$, $-1$, $\beta$ or $-\beta$. Since $p$ is not Exceptional $A$, $\pi t^{3}$ cannot be congruent to $1$ or $-1\bmod{9}$, and we get a contradiction.
\qed
\end{proof}

\begin{remark*}[1]
The same need not be true for $u\pi$ or $v\pi$. Suppose for example that $p=19$. Take $\pi$ to be $5u+2v$. Then $\pi\equiv 1\mod{3}$ and Theorem \ref{theorem2.4} tells us that $\pi$ is not a sum of two cubes. But $u\pi = -2u+3v = -8u+(6u+3v) = -8u+v\beta^{3}$, and Theorem \ref{theorem1.1} tells us that $u\pi$ is a sum of two cubes. Similarly $v\pi = -3u-5v = (-3u-6v)+v = u\beta^{3}+v$; it follows that $v\pi$, too, is a sum of two cubes. I don't know what happens in general, and would be grateful if someone with computer skills could look at the problem.
\end{remark*}

\begin{remark*}[2]
One can't dispense with the condition that $p$ is not Exceptional $A$. Suppose for example that $p=73$; take $\pi$ to be $1+9u$, which is $1\bmod{9}$. But now $\pi=8u+(1+u)=8u-v$, and Theorem \ref{theorem1.1} shows that $\pi$ is a sum of two cubes in \newk. The first five $p\equiv 1\mod{9}$ that are Exceptional $A$ are $73$, $271$, $307$, $523$ and $577$.  For each of them the corresponding $\pi\equiv 1\mod{3}$ are sums of two cubes---in some cases I can write down such representations of $u\pi$ and/or $v\pi$. But the general situation seems murky once again.
\end{remark*}

There have been developments in the sum of two cubes problems since 2020 when \cite{1} gave a summary of known results.

Kriz now has an article online at arXiv where he indicates a proof of the remaining part of Sylvester's conjecture---if $p$ is a prime $\equiv 8\mod{9}$ then there are infinitely many ways of writing $p$ and $p^{2}$ as sums of two rational cubes.

Also, it had been suggested that for large $N$, roughly half of the cube-free integers between $0$ and $N$ are sums of two rational cubes. Progress is being made towards proving results of this sort.  All this is far beyond my understanding, but I find it astonishing.



\end{document}